\documentclass[twoside]{article}
\usepackage{authblk}
\usepackage{amssymb,verbatim}
\usepackage{amsmath}
\usepackage{amsthm}
\usepackage{graphicx}
\usepackage{color}

\usepackage[all]{xy}

\makeatletter
\newcommand{\subjclass}[2][1991]{%
  \let\@oldtitle\@title%
  \gdef\@title{\@oldtitle\footnotetext{#1 \emph{Mathematics subject classification.} #2}}%
}
\newcommand{\keywords}[1]{%
  \let\@@oldtitle\@title%
  \gdef\@title{\@@oldtitle\footnotetext{\emph{Key words and phrases.} #1.}}%
}

\newtheorem{theorem}{Theorem}[section]
\newtheorem{proposition}[theorem]{Proposition}

\newtheorem{Lemma}[theorem]{Lemma}
\theoremstyle{definition}
\newtheorem{Definition}[theorem]{Definition}
\theoremstyle{remark}
\newtheorem{remark}[theorem]{Remark}

\newcommand{\cO}{{\mathcal O}}
\newcommand{\cH}{{\mathcal H}}

\newcommand{\cL}{{\mathcal L}}

\newcommand{\Z}{\mathbb Z}
\newcommand{\C}{\mathbb C}

\newcommand{\proj}{\mathbb P}
\newcommand{\pu}{\mathbb P ^1}

\newcommand{\f}{\varphi}
\newcommand{\ra}{\rightarrow}

\DeclareMathOperator{\Sing}{Sing}
\DeclareMathOperator{\Hilb}{Hilb}
\DeclareMathOperator{\Pic}{Pic}
\begin{document}

\title{Fano's Last Fano}

\author[$\star$]{Marco Andreatta}
\author[$\dagger$]{Roberto Pignatelli}

\affil[$\star$]{Dipartimento di Matematica, Università di Trento, Italia; marco.andreatta@unitn.it}
\affil[$\dagger$]{Dipartimento di Matematica, Università di Trento, Italia; roberto.pignatelli@unitn.it}

\keywords{Fano manifolds, Grassmannians, Hilbert Schemes

\medskip \noindent
{{\bf Acknowledgments}.
Besides Gino Fano we like to thank: Elena Scalambro, who brought to our attention the paper during the INDAM workshop {\it Algebraic Geometry Between Tradition and Future - An Italian Perspective} held in Roma in December 2021;  Ivan Cheltsov, for an enlightening conversation on Mukai's Example 2 \cite{Mu1};  Barbara Bolognesi for pointing us the useful reference \cite{BC};  Sandro Verra for wise suggestions on the Picard number of some Fano manifolds.}}

\maketitle

\begin{abstract}{ In 1949 Fano published his last paper on $3$-folds with canonical sectional curves \cite{Fa5}. There he constructed and described a $3$-fold of degree $22$ in a projective space of dimension $13$ with canonical curve section, which we like to call {\it Fano's last Fano}. We report on Fano's construction and we provide various (in our opinion missing) proofs in modern language. Moreover, we try to use results and techniques available at that time.
After that we construct Fano's last Fano with modern tools, in particular via the Hilbert scheme of zero cycles on a rational surface; as a consequence we easily point out the corresponding example in the Mori-Mukai classification \cite{MM}.}
\end{abstract}

\section{Introduction}

In the early 1900, Gino Fano started a systematic study of projective varieties of dimension $3$. His pioneering work was remarkably original and deep, although at that time the necessary mathematical tools, especially in the field of Algebra, were not well developed. It is generally accepted that his proofs are not enough rigorous for the modern standard; on the other hand, they contain many intuitions on the geometry of projective varieties, which turned out to be correct and fundamental. 

\smallskip
We consider smooth projective varieties $X$ defined over $\C$; if $n$ is the dimension of $X$, we sometime call $X$ an  {\sl $n$-fold}. We denote by $K_X$ the  {\sl canonical sheaf} of $X$.

Fano studied projective $3$-folds $X \subset \proj^N$ such that for general hyperplanes $H_1, H_2$ the curve $\Gamma:= X \cap H_1 \cap H_2$ is canonically embedded into
$H_1 \cap H_2$ (i.e. $K_\Gamma$ embeds $\Gamma$). 
Fano called them {\it Variet\`a algebriche a tre dimensioni a curve sezioni canoniche} \footnote {Algebraic Varieties of dimension $3$ with canonical curve section.}, \cite{Fa2,Fa3,Fa4,Fa5}. 

He considered this class of varieties to provide a counterexample to a Castelnuovo type rationality criteria for $3$-folds and to the L\"uroth problem. Although some of these varieties have all plurigenera and irregularity equal to zero, he understood that some of them should not be rational. It is generically accepted that none of Fano's attempts to prove nonrationality should be considered rigorous. 
The first modern and accepted proof of the nonrationality of quartic $3$-folds in $\proj ^4$ is the celebrated Iskovskikh and Manin's wheareas the nonrationality of the cubic $3$-fold in $\proj ^4$ was proved by Clemens and Griffiths, \cite{CG,IM}. B. Segre constructed  some unirational quartic $3$-fold in $\proj ^4$, \cite{Seg}; therefore these unirational but nonrational $3$-fold represent counterexamples to the L\"uroth problem in dimension $3$.

\smallskip
Fano started also a biregular classification of his varieties in the case of dimension $3$.  Starting from Fano's results a large number of mathematicians have constructed clever theories in the last 50 years, which are among the most spectacular achievements of contemporary mathematics. Initially, Fano's legacy has been taken into account in V. Iskovskikh's work, as well as in V. Shokurov's work, and soon after in that of S. Mori and S. Mukai. The theory of Minimal Models developed by S. Mori gave an enormous impulse; on the one hand the Minimal Model Program changed the approach to the classification of projective varieties and on the other hand the varieties studied by Fano had a central role in the classification.

 \medskip
 The following proposition is a well-known result; we provide a proof for reader's convenience.
 \begin{proposition}
 \label{modFano}
Let $X \subset \proj^N$ be a projective $n$-fold and let $H := \cO_{\proj^N}(1)_{|X}$ be the hyperplane bundle on $X$. Assume for general hyperplanes $H_1, H_2, ..., H_{n-1} \in |H|$ the curve $\Gamma:= H_1 \cap H_2\cap ...\cap H_{n-1}$ is a canonically embedded curve of genus $g$. 
Then $-K_X = (n-2)H$. 

In particular, if $n=3$ the linear system $|-K_X]$ embeds $X$  as a $3$-fold of degree $2g-2$ into projective space of dimension $g+1$,  i.e. $X: = X^{2g-2}_3 \subset \proj^{g+1}$.
\end{proposition}
  
  \begin{proof}
 Let $S := H_1 \cap H_2\cap ...\cap H_{n-2}$ be a general surface section. Denote by $\Gamma = S\cap H_{n-1}$ a curve section. For $m \geq 0$ consider the exact sequence
  
  $$0 \ra \cO_S(m-1) \ra \cO_S(m) \ra \cO_{\Gamma}(m) \ra 0$$
  and the corresponding long exact cohomology sequence
  
 $$0 \ra H^0(\cO_S(m-1)) \ra H^0(\cO_S(m)) {\xrightarrow {\alpha}} H^0(\cO_{\Gamma}(m)) \ra H^1(\cO_S(m-1))  \ra $$
 
 $${\xrightarrow {\beta}} H^1(\cO_S(m)) \ra H^1(\cO_{\Gamma}(m)) \ra H^2(\cO_S(m-1)) \ra H^2(\cO_S(m))  \ra0.$$
 
 For $m \geq 1$ the map $\alpha$ is onto since a canonically embedded curve, $\Gamma$, is projectively normal (this is a classical result attributed to Noether and Enriques-Petri). Therefore $\beta$ is injective and by decreasing induction on $m$ we have 
 $$H^1(S, \cO_S(m)) =0, \ \ \  m \geq 0.$$
 
 For $m>1$ we have, by Serre duality, $H^1(\Gamma, \cO_{\Gamma}(m))= H^0(\Gamma, \cO_{\Gamma}(1-m)) ^{\vee}= 0$;
 again by decreasing induction, we get 
 $$H^2(S, \cO_S(m)) =0, \ \ \ m > 0.$$
 
  For $m=1$ we have $\dim H^1(\Gamma, \cO_{\Gamma}(1))= 1$, which gives
  $$\dim H^2(S, \cO_S)) = 1.$$
  
  \smallskip
  Since $\dim H^2(S, \cO_S)) = 1$, by Serre duality $K_S$ is effective or trivial.
   By adjunction formula we have
  $$(K_S + \Gamma)^.\Gamma = K_{\Gamma} = \Gamma ^. \Gamma.$$
  hence ${K_S }^. \Gamma=0$. By Kleiman's criterium for the ampleness (of $\Gamma$) this implies that $K_S =0$. Since $H^1(S, \cO_S) =0$, $S$ is a K3 surface.
  
  \medskip
  Adjunction formula on $X$ implies the following
  $$(K_X + (n-2)H) ^. S = K_S =0.$$
  By Weil's equivalence criterium (\cite{Wei}, p.111, Th. 2) this implies $-K_X = (n-2)H$.  
  
  \smallskip
  The computations for the case of $3$-folds are quite straightforward.
 
  \end{proof}

Nowadays we define a Fano manifold as follows.
\begin{Definition} 
\label{Fman}
A smooth projective variety $X$ is called a {\it Fano manifold} if $-K_X$ is ample.

The  {\it index of $X$} is defined as the greatest integer $r$ such that  $- K_X = -rL$ for a line bundle $L$. The ample line bundle $L$ which achieves the maximum is called the {\it fundamental line bundle (or divisor)}.

If $Pic(X) = \Z$, then $X$ is called a  {\it Fano manifold of the first species} or a {\it prime Fano manifold}. 
\end{Definition}

\begin{remark} 
By Proposition \ref{modFano}, the varieties considered by Fano are Fano manifolds with $-K_X$ very ample and index $n-2$. It is straightforward to check that Fano manifolds with $-K_X$ very ample and index $n-2$ have canonical curve section.

The slightly more general definition with $-K_X$ simply ample was not on hand at that time and it is absolutely appropriate for modern taste and techniques. 
\end{remark}

\smallskip
V.A.  Iskovskikh, \cite{Isk1} and \cite{Isk2}, has taken up the classification. By modern tools he has been able to justify and generalize Fano's work, thus obtaining a complete classification of {\it prime Fano $3$-folds}. 
If $g$ is the genus of the curve section, he proved that $3 \leq g \leq 12$ and $g \not= 11$. For every such $g$ he gave a satisfactory description of the associated Fano variety. He used the Fano's method of double projection from a line. In particular he needed the existence of a line and the existence of smooth divisor in the linear system $|-K_X|$. These are delicate results proved later by Shokurov in \cite{Sh1} and \cite{Sh2}. 

Among his results, a nice one is the construction of a prime Fano manifold $X^{22}_3  \subset \proj ^{13}$, which apparently was omitted by Fano. Some years later, S. Mukai gave a new method to classify prime Fano-Iskovskikh $3$-folds based on vector bundle constructions, \cite{Mu1}, providing a new description of $X^{22}_3 \subset \proj ^{13}$ (see also \cite{MU}).

In the same period, S. Mori and S. Mukai \cite{MM} gave a classification of all Fano $3$-fold with Picard number greater than or equal to $2$ (i.e. not prime), and they finished the classification of Fano $3$-fold. Their classification is based on the Iskovskikh's and on the Mori Theory of extremal rays, via the so called "two rays game".

Fano manifolds of any dimension $n$ and index $r \geq (n-2)$ were classified by Kobayashi and Ochiai (\cite{KO}, index $n$), T. Fujita ( \cite{Fu}, index $(n-1)$) and S. Mukai; the latter classified all Fano manifolds of index $(n-2)$ under the assumption that the fundamental divisor has an effective smooth member, \cite{Mu1}. Later on, M. Mella proved that this assumption is always satisfied \cite{Me}.

\medskip
Let us now briefly describe the purpose and the content of our paper.  In 1949, Fano published in {\it Rendiconti dell'Accademia dei Lincei} his last paper on $3$-folds with canonical sectional curves, under the title {\it Su una particolare variet\`a a tre dimensioni a curve-sezioni canoniche} \footnote{On a special $3$-fold with canonical curve section} \cite{Fa5}. At the time he was 78 years old and he died three years afterwards.
In the paper, he constructed and described a $3$-fold of the type $X^{22}_3 \subset \proj^{13}$ with canonical curve section, which we like to call Fano's last Fano. At first, we even thought that this was the variety which was missing in his classification, as claimed by Iskovskikh; very soon we realized that this variety is not {\it prime}, i.e. it has Picard rank $2$. Therefore it is not isomorphic to either the Iskovskikh nor to the Mukai example and it should be searched in the Mori-Mukai classification.

Fano's paper was almost never quoted after its publication and it has been long ignored by most modern mathematicians. Very likely, this is due to the fact that L. Roth cited the paper on page $93$ of his book {\it Algebraic Threefolds} (1955). He wrote that Fano examined {\it a particular fourfold of the third species ...}; probably Roth read the paper too quickly and did not realize that Fano was actually searching for a $3$-fold and not (only) for a $4$-fold.

\smallskip
In Section 2, we will report on Fano's construction, using his own words in Italian, with our translation in English. His arguments are correct but very often without a complete proof. We provide diverse proofs in modern language and try to use results and techniques available at that time in order to support Fano's correctness. We hope that the reader might enjoy, as we did, the beauty as well as the elegance and simplicity of Fano's example.

In Section 3, we construct Fano's with modern tools, in particular via the Hilbert scheme of zero cycles on a rational surface. As a consequence we can easily point out the corresponding example in the Mori-Mukai list. Within the modern description the reader can easily derive the properties of the example studied in the second part of Fano's paper, for instance its rationality and the description of ruled sub-surfaces.

\section{Fano's Construction of a $X^{22}_3 \subset \proj ^{13}$,}

\subsection{Construction and smoothness}

We report and comment on the first section of Fano's paper of 1949, \cite{Fa5}, using {\it verbatim} Fano's words in Italian, providing an English translation in the footnotes. \\
{\it Ho incontrato recentemente una variet\`a a tre dimensioni a curve-sezioni canoniche, che naturalmente appartiene alla serie delle $M_3^{2p-2}$ di $S_{p+1}$ (qui  $p=12$),
oggetto di mie ricerche in quest'ultimo periodo, ma non ha finora richiamata particolare attenzione. Ne dar\`o qui un breve cenno.}
\footnote {Recently, I discovered a $3$-dimensional variety with canonical sectional curves, which naturally belongs to the collection of $M_3^{2p-2}$ of $S_{p+1}$ (here $p=12$), which was the topic of my research in this last period, but which up to now has not drawn special attention. I will give here a brief mention}

Let us explain Fano's notation. $S_{n}$ is what we denote now with $P_{\C}^{n}$ whereas $M_m^{d}$ is a subvariety of $P_{\C} ^{n}$ of dimension $m$ and degree $d$ (note that in what follows he uses the world {\it ordine} for degree). Therefore, in the notation of the previous section, he discovered a $X^{22}_3 \subset \proj ^{13}$ of dimension $3$ with canonical sectional curves. We proceed with his notation.

\smallskip
{\it Consideriamo nello spazio $S_5$ una rigata razionale normale $R^4$ (non cono),
che per semplicit\`a supponiamo del tipo pi\`u generale, cioè con $\infty ^1$ coniche direttrici irriducibili;  
e con essa la variet\`a $\infty ^4$ delle sue corde. Quale ne \`e l'immagine $M_4$ nella Grassmanniana 
$M_8^{14}$ di $S_{14}$ delle rette di $S_5$?}
\footnote {Let us consider in the space $S_5$ a normal rational ruled surface $R^4$ (not a cone), which for simplicity we suppose of general type, that is, with $\infty ^1$ irreducible conics as ruling; with it we consider the variety $\infty ^4$ of its chords. What is the image of $M_4$ in the Grassmannian $M_8^{14}$ in  $S_{14}$ of lines in $S_5$?}

\smallskip
The ruled surface $R^4$ is viewed as the image of $\pu \times \pu$ embedded in $\proj^5$ by the complete linear system $|(1,2)|$. It is rational and it has degree $4$. It is normal because the general hyperplane section is a normal rational curve of degree $4$ in $\proj ^4$, the image of general element of $|(1,2)|$. 
Thus, the surface $R^4$ has two rulings: one is given by the {\it lines} contained in $R^4$, the images of the divisors in the  complete linear system $|(0,1)|$, and one given by the {\it conics} mentioned by Fano,  which are all irreducible conics in $R^4$, the images of the divisors in the  complete linear system $|(1,0)|$. 

\smallskip
Consider the Grassmannian of the lines in $\proj ^5$ embedded via the Pl\"ucker embedding and following Fano denote it as $M^{14}_8 \subset S_{14}$. To explain Fano's notation note that it is a compact complex manifold of dimension $8$ which is embedded in a projective space of dimension $\frac{5\cdot 6}2-1=14$. Its degree as a subvariety  can be computed by standard Schubert calculus, namely it is equal to $\sigma_1^8=1+3^2+2^2=14$. 

The variety $M_4$ is defined by Fano as the subset of  the Grassmannian of lines in $\proj ^5$ given by the {\it chords} of $R^4$. Since we are looking for a complete variety, we need to interpret the word "corde" in a broad sense, that is secant and tangent lines. 

At this point, it seems to us that Fano gives for granted that $M_4$ is a smooth irreducible variety of dimension $4$.  This is in our opinion not  obvious, in particular its smoothness. For these purposes we formulate the proposition below and dedicate some pages at its proof.  In the next section, we provide a second, more geometric, proof.

\begin{proposition} 
\label{prop: M4 is smooth}
$M_4 \subset M_8^{14}$ is an irreducible smooth variety of dimension $4$.
\end{proposition}

For the proof we need to set up some notation and preliminaries.
We take homogenenous coordinates $([x_0,x_1],[y_0,y_1])$ on $\pu \times \pu$ and we denote by $\pi_x, \pi_y \colon \pu \times \pu \rightarrow \pu$ the natural projections, namely:
\begin{align*}
\pi_x([x_0,x_1],[y_0,y_1])=&[x_0,x_1],&
\pi_y([x_0,x_1],[y_0,y_1])=&[y_0,y_1].&
 \end{align*} 

We take coordinates $[z_{ij}]$ in $\proj ^{5}$ so 
that the embedding $\pu \times \pu \rightarrow R^{4} \subset \proj ^{5} $ is given by identifying $ z_{ij}=x_i y_0^{2-j} y_1^j$, that is: 

\begin{equation}
\label{eq:z=xy}
\begin{aligned}
z_{00}=&x_0y_0^2,&
z_{01}=&x_0y_0y_1,&
z_{02}=&x_0y_1^2,\\
z_{10}=&x_1y_0^2,&
z_{11}=&x_1y_0y_1,&
z_{12}=&x_1y_1^2.
\end{aligned}
\end{equation}

\smallskip
$R^4$ is not the Veronese surface, therefore secant and tangent lines cover all $\proj^5$
(Severi \cite{Sev} proved that the Veronese surface, that is $\proj^2$ embedded by $\cO(2)$,  is the only surface in $\proj ^5$ whose secant and tangent lines do not cover $\proj^5$).

\smallskip
Consider the $3$-fold $G_3^3 \subset \proj ^5$ defined by the condition
\[
rk
\begin{pmatrix}
z_{00}&z_{01}&z_{02}\\
z_{10}&z_{11}&z_{12}\\
\end{pmatrix}=1.
\]

We claim that $G^3_3$ contains $R^4$. More precisely if we consider the ruling by conics of $R^4$, $G_3^3$ is the union of the planes containing these conics. It is obvious that for each point in $G^3_3$ there are at least two (in fact infinitely many) secants  to $R^4$: the point, say $p$, is contained in a plane, say $\pi$, in $G_3^3$ containing a conic of the ruling by conics of $R^4$ and thus all the lines through $p$ contained in $\pi$ are secants to $R^4$. We show that the converse is also true.

\begin{Lemma}
\label{uniquesec}
For each point $p \not \in G^3_3$ there is a unique secant of $R^4$ through it.

In particular for each point on a general line $l \subset \proj ^5$ (i.e. disjoint from $G^3_3$) there is a unique secant to $R^4$ through it.
\end{Lemma}

\begin{proof}
The existence follows from the above quoted result of Severi \cite{Sev} since $R^4$ is not the Veronese surface.

To prove uniqueness, we argue by contradiction. Assume there are two distinct secants lines through $p$, say $s_1 \neq s_2$. Let $\pi$ be the plane spanned by $s_1$ and $s_2$.
The hyperplanes containing $\pi$ define a codimension $3$ subsystem $\Gamma$ of $|\cO_{\pu \times \pu}(1,2)|$. 

$\Gamma$ does not have fixed components: assume by contradiction that $\Gamma$ has a fixed component (with reduced part) $F$. By dimension reasons, $F$ has to be contained in $|\cO_{\pu \times \pu}(1,0)|$ or in $|\cO_{\pu \times \pu}(0,1)|$, since the movable part has dimension $2$. $F$ can not be in $|\cO_{\pu \times \pu}(1,0)|$, since the plane $\pi$ coincides with the intersection of all hyperplanes through an irreducible conic in $R^4$, so that $p$ is contained in $G^3_3$, i.e. a contradiction. 
$F$ can not be in $|\cO_{\pu \times \pu}(0,1)|$, because the plane $\pi$ contains a line of $R^4$, say $r$. By assumption $r$ is neither $s_1$ nor $s_2$, so each $s_j$  intersects $R^4$ in at least one point $p_i$ out of $r$. On the other hand, there is no plane in $\proj^5$ containing a line and two more distinct points all contained in $R^4$, so $p_1=p_2$. Then both $p_i$ coincide with $p$ contradicting the assumption  $p \not\in R^4$.

Since  $\Gamma$ has no fixed components, there is a hyperplane containing both $s_1$ and $s_2$ that cuts an irreducible curve of $R^4$; however any irreducible element of $|\cO_{\pu \times \pu}(1,2)|$ is smooth so
this is then a smooth rational normal quartic $C^4$. Both $s_1$ and $s_2$ are secants to it, cutting respectively subschemes $\delta_1$ and $\delta_2$ of length at least two on $C^4$. Since  $p \not\in R^4$, $\delta_1 \cap \delta_2=\emptyset$ and therefore the plane  $\pi$ gives a pencil in $|\cO_{C^4}(1,2) \cong \cO_{\pu}(4)|$ with fixed locus of degree at least $4$, a contradiction.
\end{proof}

\begin{proof} [Proof of Proposition \ref{prop: M4 is smooth}]
Consider the rational map
\[
(\pu \times \pu)^2 \dashrightarrow M_4 \subset M^{14}_8
\]
associating, to each pair $(p,q)$ of distinct points $p\neq q \subset \pu \times \pu$, the unique line $\overline{pq}$ through them. 

This is a dominant map, hence $M_4$ is irreducible.
 
 The claim about the dimension  $\dim M_4=\dim (\pu \times \pu)^2 = 4$ follows since this map is generically finite. To show this take  a secant line $r$ of $R^4$ intersecting $R^4$ transversally; for example the line $z_{01}=z_{02}=z_{10}=z_{11}$ which intersects $R^4$ trasversally in $[1,0,0,0,0,0]$ and $[0,0,0,0,0,1]$, the images of the points $([1,0],[1,0])$ and $([0,1],[0,1])$ of $\pu \times \pu$. Then $r\cap R^4$ is a finite set $p_1,\ldots, p_l$ of $l \geq 2$ points and the pre image of $r$ in $(\pu \times \pu)^2$ is the set of pairs $(p_i,p_j)$ with $i\neq j$, finite as well.
 
\smallskip
From now on we prove smoothness.
We first show that the action (of possibly a subgroup) of the automorphism group of the variety $M_4$ splits it in finitely many orbits; this reduces our claim to finitely many local computations. 

Consider the standard action of $\proj GL_2(\C)$ on $\pu$ and its standard {\it linearization}, the action of  $GL_2(\C)$ on the coordinate ring of $\proj^1$, $\C[x_0,x_1]=\bigoplus_d H^0(\cO_{\pu}(d))$.
This associates to each matrix a ring homomorphisms  as follows:
\[\begin{pmatrix}
a_{11}&a_{12}\\
a_{21}&a_{22}
\end{pmatrix}\colon \
\alpha_0 x_0+\alpha_1 x_1\mapsto (a_{11}\alpha_0+a_{12}\alpha_1)x_0+(a_{21}\alpha_0+a_{22}\alpha_1)x_1.
\]

Taking two copies of this action, we get an action of $GL_2(\C)^2$ on $\pu \times \pu$ and a linearization of it to each $H^0(\cO_{\pu \times \pu}(d_1,d_2))$; in particular on $H^0(\cO_{\pu \times \pu}(1,2))$, which we identify via  \eqref{eq:z=xy} to $H^0(\cO_{\proj ^5}(1))$. The naturally induced action of $GL_2(\C)^2$ on $\proj ^5$ preserves $R^4$, inducing on it the action (on $\pu \times \pu$) we started with.
Finally, the action on $\proj^5$ induces an action on the Grassmannian  of the lines in $\proj^5$ that preserves $M_4$. Therefore we can define an action of  $GL_2(\C)^2$  on $M_4$.

We show now that this action splits $M_4$ in exactly $5$ orbits.

We say that two distinct points $p\neq q$ of $R^4$ are {\it in general position with respect to the rulings} if $p$ and $q$ belong to two different lines and to two different irreducible conics in $R^4$.  

For every element $g \in GL_2(\C)^2$, if two points $p$ and $q$ are in general position with respect to the rulings, then their images $gp$ and $gq$ are in general position with respect to the rulings as well. In particular the orbit of $\overline{pq} \in M_4$ is made by secants on two distinct points in general position with respect to the rulings.
Since the action of $GL_2(\C)$ on $\pu$ is $2-$transitive \footnote{In fact it is $3-$transitive, but we only need $2-$transitivity here} we conclude that the Zariski open subsets of $M_4$ of the secants on two points in $R^4$ in general position with respect to the rulings form an orbit. 

A similar argument shows that the secants through two distinct points  belonging to the same irreducible conic of $R^4$ form a second orbit of dimension $3$. A third orbit is obtained by considering the secants through two distinct points  belonging to the same line of $R^4$: these are the lines contained in $R^4$, forming an orbit of dimension $1$.

Considering the lines that are tangent to $R^4$ we get similarly three orbits. However, the tangent to a point in the direction of the line through it coincides with the line itself, thus giving an orbit that has been already considered. 

\smallskip
Summing up, $GL_2(\C)^2$ decomposes $M_4$ in $5$ orbits as follows, the notation for each stratum $r_d$ is settled so that $d$ is its dimension:

\begin{itemize}
\item[$0_4$] the secants through two points in general position with respect to the rulings;
\item[$0_3$] the tangents to a point $p$ in a direction different from both the directions of the line and of the conic through it;
\item[$2_3$] the secants through two points $p \neq q$ that belong to the same irreducible conic;
\item[$2_2$] the lines tangent to any irreducible conic contained in $R^4$;
\item[$1_1$] the lines contained in $R^4$.
\end{itemize}

\medskip
It is enough to show the smoothness of $M_4$ in one point for each orbit.
Moreover, since the smooth locus of a variety is a Zariski open subset, it is enough to show the smoothness at a point of  $2_2$ and at a point of $1_1$: in fact a neighborhood of them intersects all other orbits.

\smallskip
Let us start with a point of $1_1$.

We consider a "standard" chart for the Grassmannian  as follows. 
To each rank $2$ matrix 
\[
\begin{pmatrix}
a_{00}&a_{01}&a_{02}&a_{10}&a_{11}&a_{12}\\
b_{00}&b_{01}&b_{02}&b_{10}&b_{11}&b_{12}\\
\end{pmatrix}
\]
we associate the line through the points $a,b\in \proj ^5$ given by its rows. In other words the coordinate $z_{ij}$ evaluated in $a$ and $b$ gives $a_{ij}$ and $b_{ij}$, respectively.

The matrices 
\[
\begin{pmatrix}
1&a_{01}&a_{02}&0&a_{11}&a_{12}\\
0&b_{01}&b_{02}&1&b_{11}&b_{12}\\
\end{pmatrix}
\]
give a standard chart of the Grassmannian of the lines in $\proj^5$, an open subset parametrized by the affine coordinates $a_{01},a_{02},a_{11},a_{12},
b_{01},b_{02},b_{11},b_{12}$.
Note that its origin 
\[
\begin{pmatrix}
1&0&0&0&0&0\\
0&0&0&1&0&0\\
\end{pmatrix}
\] 
corresponds  to the line $z_{01}=z_{02}=z_{11}=z_{12}=0$  in $R^4$, which by \eqref{eq:z=xy} originates from the line $y_1=0$ in $\pu \times \pu$. So the origin is a point in the orbit $1_1$. 

\smallskip

A point in this chart corresponds to the line in $\proj ^5$  given by parametric equations:
\[\begin{pmatrix}
\alpha,&\alpha a_{01}+ \beta b_{01},&\alpha a_{02}+\beta b_{02},&\beta &\alpha a_{11}+\beta  b_{11},&\alpha a_{12}+\beta  b_{12}\\
\end{pmatrix}.
\]
with parameters $\alpha,\beta$.
The rational surface $R^4$ is contained in the following three hyper quadrics in $\proj ^5$ :
\begin{align*}
z_{00}z_{11}&=z_{01}z_{10},&z_{00}z_{02}&=z_{01}^2,&z_{10}z_{12}&=z_{11}^2.
\end{align*}

Therefore, if a point in the above parametrized line belongs to $R^4$, it will a fortiori belong to each of these hyper quadrics, that is
\begin{equation}\label{eqn: r=2}
\alpha(\alpha a_{11}+\beta  b_{11})=\beta(\alpha a_{01}+ \beta b_{01}),
\end{equation}
\begin{equation}\label{eqn: A}
\alpha(\alpha a_{02}+\beta b_{02})=(\alpha a_{01}+ \beta b_{01})^2 ,
\end{equation}
\begin{equation}\label{eqn: B}
\beta(\alpha a_{12}+\beta b_{12})=(\alpha a_{11}+ \beta b_{11})^2 .
\end{equation}

These are three homogeneous equations of degree $2$ in the variables $\alpha,\beta$.
If the line  is secant (or tangent) then these three homogeneous equations define a scheme of $\pu$ of length at least $2$. Therefore the equations must be pairwise proportional.  This will allow to express the four variables $a_{ij}$ as holomorphic functions in terms of the other four as follows.

Assume $a_{11}\neq 0$, $b_{01} \neq 0$. The coefficient of $\alpha^2$ in \eqref{eqn: B} is equal to  $a_{11}^2$. Therefore, in order to have \eqref{eqn: r=2} proportional to \eqref{eqn: B} we have to multiple it by $-a_{11}$. Comparing the other coefficients, we get the following two conditions:
\begin{align*}
2a_{11}b_{11}-a_{12}=&a_{11}(b_{11}-a_{01})&
b_{12}-b_{11}^2&=a_{11}b_{01}&
\end{align*} 

Similarly, looking at the coefficient of $\beta^2$, we see that \eqref{eqn: A} equals \eqref{eqn: r=2}  multiplied by $b_{01}$. Comparing the other coefficients, we get
\begin{align*}
2a_{01}b_{01}-b_{02}=&b_{01}(a_{01}-b_{11})&
a_{02}-a_{01}^2&=a_{11}b_{01}&
\end{align*} 

We use these four equations to express the variables $a_{02},b_{02},a_{12},b_{12}$ as holomorphic functions of the others. Hence we obtain a smooth parametrization given by the matrices of the form
 \[
\begin{pmatrix}
1&a_{01}&a_{01}^2+a_{11}b_{01}&0&a_{11}&a_{11}(a_{01}+b_{11})\\
0&b_{01}&b_{01}(a_{01}+b_{11})&1&b_{11}&a_{11}b_{01}+b_{11}^2\\
\end{pmatrix}.
\]

Note that the parametrization above is an embedding of $\C^4$ in  $M^{14}_8$, thus giving a smooth $4-$dimensional manifold. By construction, this manifold contains the intersection of $M_4$ with our chart
 \footnote{We have parametrized the secants to the complete intersection of three quadrics, which is a variety that contains $R^4$ but does not coincide with it. Hence the parametrized locus contains the intersection of $M_4$ with the chosen chart. }.
Since $M_4$ is irreducible of dimension $4$, they coincide. This proves the smoothness of $M_4$ at a point of $1_1$, the origin of this chart.

\smallskip

Finally, consider a point in $2_2$, a line tangent to a conic in $R^4$, for example the line $z_{02}=z_{10}=z_{11}=z_{12}=0$.
By the same techniques used in the previous case, the reader can show that a neighbourhood of this point in $M_4$ is the smooth manifold given by the matrices 
\[
\begin{pmatrix}
1&0&a_{02}&a_{10}&a_{02}b_{10}&a_{02}(a_{10}+b_{10}b_{02})\\
0&1&b_{02}&b_{10}&a_{10}+b_{02}b_{10}&a_{10}b_{02}+b_{10}a_{02}+b_{10}b_{02}^2\\
\end{pmatrix}.
\]
\end{proof}

\subsection{Computing the degree}
Fano proceeds to compute the degree of $M_4$ with respect to the Pl\"ucker embedding of the Grassmannian.

{\em Determiniamo anzitutto l'ordine di questa $M_4$, ad esempio l'ordine della
superficie sua intersezione con un $S_{12}$, vale a dire della $\infty ^2$
di rette comune alla $\infty ^4$  suddetta e a due complessi lineari. Valendoci di due complessi costituiti
risp. dalle rette incidenti a due $S_3$, questi ultimi contenuti in un $S_4 \equiv \sigma$ e aventi
perci\`o a comune un piano $\pi$, la  $\infty^2$ di rette in parola si spezzer\`a nei due sistemi
delle corde di $R$ contenute in $\sigma$ e di quelle incidenti al piano $\pi$.}
\footnote {Let's first determine the order of this $M_4$, for instance the order of the
surface which is the intersection with an $S_{12}$, that is the order of the $\infty^2$  common lines 
of the  above $\infty^4$ and two linear complexes.
Making use of two constituted complexes
resp. from the incident lines to two $S_3$, both contained in an $S_4 \equiv \sigma$ and having
therefore a common plane $\pi$, the $\infty^2$ of straight lines in question will break in the two systems
of the chords of $R^4$ contained in $\sigma$ and of those incident to the plane $\pi$.}

Fano's strategy is to compute the degree of $M_4$ as the degree of a surface obtained cutting $M_4$ by two hyperplanes sections in $M_8^{14}$. He chooses two very special hyperplanes, given by the lines intersecting two linear subspaces of dimension $3$ in $\proj ^5$ which are in "special position". Namely, these two $\proj ^3$ in $\proj ^5$ intersect along a plane $\pi$; equivalently both of them are contained in a hyperplane $\sigma \subset \proj ^5$. 

He notices that the lines in the intersection of the two hyperplanes in $M_8^{14}$ are exactly the lines contained in $\sigma$ and the lines intersecting $\pi$. In fact, on one side a line intersecting $\pi$ in a point $p$ intersects both $\proj ^3$ in $p$; and also a line contained in $\sigma$ must intersect both $\proj ^3$, which are in $\sigma$.
On the other side, if a line intersects both $\proj ^3$ and does not intersect $\pi$, it intersects them in distinct points; the line therefore contains two distinct points of $\sigma$, thus it is contained in $\sigma$. 

Let us denote with $S^{\sigma}$ the subvariety of $M_4$ of the lines contained in $\sigma$ and with $S_\pi$ the subvariety of lines intersecting $\pi$. Fano uses the following formula, which we like to justify with a proof.

\begin{Lemma}\label{lem: both surfaces are generically smooth}
\[
\deg M_4 =\deg S^\sigma + \deg S_\pi 
\]
\end{Lemma}
\begin{proof}
The degree of a variety equals the degree of any hyperplane section, in particular the degree of $M_4$ is equal to the degree of its intersection with the above two (special) hyperplanes. So far, at the moment we have only proved that this intersection coincides with $S^\sigma \cup  S_\pi$ set-theoretically. We need to prove that the intersection is reduced or, equivalently, that there is a choice of the two $\proj ^3$s such that the linear section is smooth in (at least) a point of $S_\pi$ and a point of $S^\sigma$. 

Let us prove it in coordinates. Take the two hyperplane sections giving the secant lines which intersect $\{z_{01}-z_{10}=z_{00}=0\} \hbox{   and } \{ z_{01}-z_{10}=z_{11}=0 \}$ respectively, 
so that $\sigma=\left\{ z_{01}-z_{10}=0\right\}$ and $\pi=\left\{ z_{01}-z_{10}=z_{11}=z_{00}=0 \right\}$.

In the chart near a point of type $1_1$ studied in the proof of Proposition \ref{prop: M4 is smooth}, the two hyperplane sections are defined respectively by $b_{01}=1$ and $a_{01}b_{11}=a_{11}(b_{01}-1)$. Their intersection is the locus $b_{01}-1=a_{01}b_{11}=0$. It is smooth at the general point of both components,  namely $S^\sigma$, which is $b_{01}-1=a_{01}=0$, and $S_\pi$, which is $b_{01}-1=b_{11}=0$.
\end{proof}

\medskip
Fano computes first the degree of $S^\sigma$.

{\em Le prime sono le $\infty^2$ corde di una $C^4$ razionale normale, e nella Grassmanniana delle rette di $\sigma$
hanno per immagine una superficie $\varphi^9$ di $S_9$ di del Pezzo.}
\footnote {The first are the $\infty^2$ chords of a normal rational $C^4$, and in the Grassmannian of the lines of $\sigma$
they have as image a del Pezzo surface  $\varphi^9$ of $S_9$}

 If $\sigma$ is general, $R^4 \cap \sigma$ is a rational normal curve of degree $4$ in $\sigma = \proj ^4$, we denote it by $C^4$. The lines contained in a hyperplane $\sigma \subset \proj ^5$ are mapped by the Pl\"ucker  embedding into a linear $\proj ^9 \subset \proj ^{14}$. Since $S^\sigma$ is the subvariety of $M_4$ of the lines contained in $\sigma$, it is contained in that $\proj ^9$. Fano's claim can be formulated in the following way. 

\begin{Lemma}
\label{lem: DP9}
The surface $S^\sigma$ is embedded in $\proj^9$ as a del Pezzo surface of degree $9$. 
\end{Lemma}
\begin{proof}
The subvariety $S^\sigma$ in $\proj^9$ is by construction the image of the map from $C^4 \times C^4 \ra \proj^9$ which maps a pair $(p,q)$ on the point of the Grassmannian of the lines in $\proj ^4$ corresponding to the secant $\overline{pq}$, embedded in $\proj^9$ by the standard Pl\"ucker embedding.

It factors through  the second symmetric product of $C^4\cong \pu$, which is isomorphic to $\proj ^2$ as shown by the degree $2$ map $\pu \times \pu \rightarrow \proj^2$  
\[
([x_0,x_1],[y_0,y_1]) \mapsto [x_0y_0,x_0y_1+x_1y_0,x_1y_1]
\]
A hyperplane section in $\proj^9$ pulls back on $\proj^2$ to a divisor in $|\cO_{\proj ^2}(d)|$ and on $\pu \times \pu $ to $|\cO_{\pu \times \pu}(d,d)|$ for some positive integer $d$.
In order to compute $d$ consider the hyperplane section given by the secants of $C^4$ intersecting a fixed general plane $\pi \subset \sigma$ and its pullback $H$ to  $\pu \times \pu$. The intersection of $H$ with  $\{p\} \times C^4 \cong C^4$ is the set of the points $q \in C^4$ such that $\overline{pq}$ is a secant to $C^4$ intersecting $\pi$. Since $H \in |\cO_{\pu \times \pu}(d,d)|$, $d$ equals the number of secant lines through a general point $p\in C^4$  intersecting $\pi$.
Choose $p \not\in \pi$ and take the projection $f_p \colon \sigma \dashrightarrow \proj^3$. The secant lines through $p$ intersecting $\pi$ are projected to the points of the plane $f_p(\pi)$ intersecting the rational normal cubic $f_p(C^4)$, so there are exactly $3$ of them, i.e. $d=3$.

We proved that $S^\sigma$ is the image of a map from $\proj^2$ on $\proj^9$ given by cubics. Note that $S^\sigma$ is not contained in any hyperplane, otherwise, by contradiction,  this would imply that $C^4$ is contained in a hyperplane of $\proj ^4$. 
Therefore $S^\sigma$ is the image by the full linear sistem $|\cO_{\proj ^2}(3)|$, defining the del Pezzo surface of degree $9$.
\end{proof}

\smallskip
Let us now interpret Fano's argument to compute the degree of $S_\pi$.

{\em  Della seconda $\infty^2$ prendiamo l'intersezione con un ulteriore complesso lineare, anche con un 
$S_3 \equiv \tau$ direttore incontrante $\pi$ in una retta. Si ha una rigata composta di una parte luogo 
delle corde di $R$ contenute nello spazio $S_4 \equiv \tau \pi$ e incidenti a $\pi$, la cui immagine \`e
sezione iperpiana di altra $\varphi_9$ di del Pezzo; e di una seconda parte luogo
delle corde incidenti alla retta $\tau \pi$.}
\footnote {Of the second $\infty^2$ we take the intersection with a further linear complex, also with $S_3 \equiv \tau$ as director, meeting $\pi$ in a line. We get a ruled surface composed of a part which is the locus of the chords of $R_4$ contained in the space $S_4 \equiv \tau \pi$ and incident to $\pi$, whose image is a hyperplane section of another del Pezzo $\varphi_9$; and of a second part, the locus of the chords incident to the line $\tau \pi$}

\smallskip
Fano considers a special codimension $2$ subspace $\tau \subset \proj^5$, special in the sense that $\tau$ intersects $\pi$ in a line. 
He takes then a special hyperplane section of $S_\pi$, the one given by the lines that meet $\tau$. By the same argument used above, this curve has two irreducible components: the secant lines in $S_\pi$  contained in the unique $\proj ^4$ generated by $\tau$ and $\pi$ and intersecting $\pi$, call it $C^{\langle \tau,\pi\rangle}_\pi$, and those intersecting the line $\tau \cap \pi$,  call it $C_{\tau \cap \pi}$. 

\begin{Lemma}\label{lem: degSpi}
For a general choice of $\pi, \tau$, the following holds:
\[
\deg S_\pi =\deg C^{\langle \tau,\pi\rangle}_\pi+C_{\tau \cap \pi}
\]
\end{Lemma}
\begin{proof}

The argument, with coordinates, is the same as that used in the proof of Lemma \ref{lem: both surfaces are generically smooth}.
We suppose again that  $\pi=\{z_{01}-z_{10}=z_{00}=z_{11}=0\}$ and we work in the same chart where we proved that  $S_\pi$ is defined by $b_{11}=b_{01}-1=0$, i.e. it is the parametrized surface  \[
\begin{pmatrix}
1&a_{01}&a_{01}^2+a_{11}&0&a_{11}&a_{11}a_{01}\\
0&1&a_{01}&1&0&a_{11}\\
\end{pmatrix}
\]
Choosing $\tau=\left\{ z_{01}-z_{10}=z_{12} \right\}$, the corresponding hyperplane section is  
\[
\det \begin{pmatrix} a_{01}-0& a_{11}a_{01}\\ 1-1 & a_{11} \end{pmatrix}=0,
\] 
which is smooth at the general point of both components,
$C^{\langle \tau,\pi\rangle}_\pi $ ($a_{01}=0$) and $ C_{\tau \cap \pi}$ ($a_{11}=0$).
\end{proof}

Fano claims that $C^{\langle \tau,\pi\rangle}_\pi$  is a hyperplane section of a del Pezzo surface of degree $9$, therefore that 
\begin{equation}\label{eq_ deg Ctp}
\deg C^{\langle \tau,\pi\rangle}_\pi =9
\end{equation}

 In fact, for a general choice, the hyperplane containing $\tau$ and $\pi$ cuts on $R^4$ a rational normal curve $C^4$ of degree $4$. As proved above  (see Lemma \ref{lem: DP9}) $C^{\langle \tau,\pi\rangle}_\pi$ is contained in the degree $9$ del Pezzo surface of its secants: more precisely it is given by those intersecting the codimension $2$ subspace $\pi$, therefore it is an hyperplane of it. 

\smallskip
To study the curve  $C_{\tau \cap \pi}$, Fano considers the union of the lines defined by it as a ruled surface $F \subset \proj ^5$. More precisely, as we have proved above (see Lemma \ref{uniquesec}), for each point on a general line $l \subset \proj ^5$, therefore also for  $\tau \cap \pi$, there is a unique secant to $R^4$ through it. Mapping each secant to its unique intersection point with $l$  one obtains a ruling of the surface $F$ to $l$.

{\em Quest'ultima rigata è di $4^o$ ordine, avendo la
retta $\tau \pi$ come direttrice semplice, e $3$ generatrici in ogni $S_4$ per essa (poiché la proiezione della rigata dalla retta $\tau \pi$ ha una cubica doppia)}
\footnote {This last ruled surface is of $4^{th}$ order, having the line $\tau \pi$ as a simple directrix, and $3$ generators in each $S_4$ for it (since the projection of the ruled surface from the line $\tau \pi$ has a double cubic)}.

Here Fano claims that $\deg F=4$. We interpret Fano's computation of the degree of this surface as follows. 

Project the surface $F$ from the line $l=\tau \cap \pi$ and obtain a curve $C\subset \proj ^3$, as each line in the ruling maps to a point. Then consider the same projection from $l$ this time  restricted to $R^4$, call it $\f_{l}: R^4 \ra \proj^3$. Since $l$ is general, it does not intersect $R^4$ and therefore $\f_l$ is a morphism. Namely a finite morphism, as every plane through $l$ intersects $R^4$ in finitely many points, or we would contradict $l \cap R^4 = \emptyset $.

We claim that $\f_{l}$ is generically injective and $C$, defined as the projection of $F$, equals the singular locus of $\f_{l}(R^4)$. In fact, we will see that $C$ is a double curve of $\f_{l}(R^4)$. 

Since $R^4$ is defined as a codimension $3$ subvariety in $\proj ^5$, a general plane through each point $p\in R^4$ intersects $R^4$ transversally only at the point $p$. Thus, for a line $l$ in this plane not passing through $p$, the projection $\f_l$ separates $p$ from any other point of $R^4$; this proves that $\f_l$ is generically injective.

Since $\f_{l}$ is a birational morphism, the image of $R^4$ is a quartic curve. Since such a curve is contained in the $3-$dimensional space parametrizing the planes containing $l$, its canonical system is trivial. The singular locus of $\f_{l}(R^4)$ is the set of the planes containing a subscheme of $R^4$ of length at least $2$. 
Any such plane contains a secant line, the unique line contained the given length $2$-scheme. Since this secant intersects $l$ (they are two lines in the same plane), it is a curve in the ruling of $F$. This shows that $\Sing \left( \f_{l}(R^4) \right) \subset C$. Conversely, every secant line $r$ to $R^4$ intersecting $l$ is contained in a unique plane, i.e. the plane spanned by $r$ and $l$, cutting the corresponding scheme of length $2$ on $R^4$.
This shows that the singular locus of $\f_{l}(R^4)$ is $C$, that is in fact a double curve.

Finally, since the canonical class of $\f_{l}(R^4)$ is trivial, by adjunction the reduced transform of the  double curve is an anticanonical divisor. Therefore, it is an element in $|\cO_{\pu \times \pu}(2,2)|$. The intersection computation $(2,2)(1,2)=4+2=6$ shows that its image has degree $\frac62=3 = \deg C$.

Therefore, we can conclude, as Fano did, that the ruled surface has degree $4$. More precisely, choose a general hyperplane $H$ containing the line $l$;  its image via $\f_l$ in $\proj ^3$ is a hyperplane that intersects the curve $C$ transversally in three points.  As a consequence, the intersection of the hyperplane $H$ with the ruled surfaces $F$ is the union of three secants passing from the three points above and the  line $l$. All together, they are $4$ lines and therefore the degree of $F$, which is equal to $\deg(F \cap H)$, is $4$.

Moreover, this will imply that
\begin{equation}\label{eq_ deg Ctau}
\deg C_{\tau \cap \pi}=4
\end{equation}

Indeed, the degree of this curve is the number of secants in $C_{\tau \cap \pi}$ in a transversal hyperplane section as, for example, the one given by the secant lines intersecting a general $\proj ^3$. As proved before, such a general projective space  $\proj^3$ intersects $F$ in $\deg F=4$ points, each belonging to one of this secants, the line in the ruling of $F$ containing it. So $ \deg C_{\tau \cap \pi}=\deg F=4$.

\medskip
Fano summarizes his computation in this way: 

{\em Complessivamente la superficie immagine delle corde di $R^4$ appoggiate a un piano \`e dunque di ordine
$9 + 4 = 13$; e la $M_4$ immagine del sistema di tutte le corde di $R$ \`e di ordine $9 + 13 = 22$.}
\footnote {
Overall the
image surface of the chords of $R^4$ intersecting a plane is therefore of order
$9 + 4 = 13$; and the $M_4$ image of the system of all the chords of $R^4$ is of order $9 + 13 = 22$}

In other worlds he claims the following:
\begin{proposition}
$M_4$ has degree $22$.
\end{proposition}

\begin{proof}
The statement follows from Lemma \ref{lem: both surfaces are generically smooth}, Lemma \ref{lem: DP9}, Lemma \ref{lem: degSpi} as well as by  \eqref{eq_ deg Ctp} and \eqref{eq_ deg Ctau}.
\end{proof}

\subsection{ $M_4$ and its general hyperplane sections are Fano} 
Fano shows that the general curve section of $M_4^{22}$ is a canonical curve of genus $12$.

{\em Le due superficie $\varphi^9$ e $F^{13}$, costituenti insieme una sezione superficiale della $M_4^{22}$, 
hanno a comune una curva sezione iperpiana della $\varphi^9$ (collo
spazio $\sigma$), perci\`o ellittica, di ordine $9$; la $M_4^{22}$ ha quindi superficie-sezioni di
genere uno, e curve-sezioni canoniche di genere $12$ (appunto $= 1 + 3 + 9 -1$)}
\footnote {
The two surfaces $\varphi^9$ and $F^{13}$, which together constitute a surface section of the $M_4^{22}$, have a common hyperplane section curve of the $\varphi^9$  (with the space $\sigma$), therefore elliptic, of order $9$; therefore the $M_4^{22}$ has surface-sections of genus one, and canonical curves-sections of genus $12$ (precisely = $1 + 3 + 9 - 1)$}.

Fano picks a hyperplane section of the surface section he had consider, namely of $S^\sigma \cup S_\pi$. Recall that in our notation $S^\sigma= \varphi^9$, and $S_\pi= F^{13}$). 

Since $S^\sigma$ is the del Pezzo surface of degree $9$, its  general hyperplane section  is a smooth plane cubic, which has genus $1$. 

We know a reducible hyperplane section of $S_\pi$, namely $C^{\langle \tau,\pi\rangle}_\pi \cup C_{\tau \cap \pi}$. Our discussion shows that $C^{\langle \tau,\pi\rangle}_\pi$ is a general hyperplane section of a del Pezzo surface of degree $9$ too, so a smooth curve of genus $1$. On the other hand, $C_{\tau \cap \pi}$ is smooth and rational, being isomorphic to $l$ by mapping each point of $l$ in the unique secant to $R^4$ through it. The intersection  $C^{\langle \tau,\pi\rangle}_\pi \cap C_{\tau \cap \pi}$ is given by the secant lines contained in the hyperplane $\langle \tau,\pi\rangle$ intersecting the line $l=\tau \cap \pi$. These are the secant lines to a general hyperplane section $R^4 \cap \langle \tau,\pi\rangle$, a rational normal curve of degree $4$, whose secants form a cubic surface. Intersecting it with a codimension $2$ linear subspace, we obtain $3$ points, and the $3$ secants through them are the intersection points of $C^{\langle \tau,\pi\rangle}_\pi$ and $ C_{\tau \cap \pi}$.  Summing up, we have hyperplane sections of $S_\pi$ formed by $2$ smooth curves of genus $0$ and $1$ respectively, which intersect in $3$ points. Hence the general hyperplane section is a smooth curve of genus $0+1+3-1=3$.

The intersection of $S^{\sigma}$ and $S_\pi$ is a hyperplane section of $S^{\sigma}$, a curve of degree $9$. The two curves obtained cutting $S^{\sigma}$ and $S_\pi$ with a general hyperplane intersect then in $9$ points. We conclude that the sectional genus of $M_4^{22}$, which is the genus of a hyperplane section of $S^\sigma \cup S_\pi$, is equal to $1+3+9-1=12$.

A general curve section of  $M_4^{22}$  is therefore a non-degenerate smooth curve of genus $12$ in $\proj ^{14-3=11}$ of degree $22$; by Riemann-Roch, this general curve section is a canonical curve, i.e. it is embedded by its complete canonical system.

 \smallskip
 To recap, we have the following Proposition.
\begin{proposition} 
The $4$-fold $M_4 =M_4^{22} \subset M_8^{14}$ is an irreducible smooth variety of dimension $4$ with canonical sectional curves. 

In particular, by Prop. \ref{modFano} it is a Fano $4$-fold of index $2$, i.e. $-K_{M_4} = 2 H$, where $H$ is the hyperplane bundle of the Grassmannian $ M_8^{14}$.
\end{proposition}

\medskip
Fano concludes by taking a very general hyperplane section $M_3$ of $M_4$. By Bertini's Theorem, this section is a smooth $3$-fold whose curve section, by construction, is canonical. Again by Proposition \ref{modFano},  this is a smooth Fano 3-fold of degree $22$.

\smallskip
{\em Le sezioni iperpiane della $M_4^{22}$  sono pertanto $M_3^{22}$ di $S_{13}$, corrispondenti al tipo
generale $M_3^{2p-2}$ di $S_{p+1}$, per $p=12$, e razionali (come risulter\`a pure dai sistemi
lineari di superficie che vi sono contenuti). Indicheremo d'ora in poi questa variet\`a
con $\mu _3^{22}$, o semplicemente $\mu$; essa \`e l'immagine del sistema $\infty ^3$ di rette $\Sigma$ 
intersezione della $\infty ^4$ delle corde di $R$ con un complesso lineare $K$ (che si supporr\`a
per ora del tipo pi\`u generale, e in posizione generica rispetto a $R^4$)}
\footnote {The hyperplane sections of the $M_4^{22}$ are therefore $M_3^{22}$ in $S_{13}$, corresponding to the general type $M_3^{2p-2}$ of $S_{p+1}$, for $p=12$, and rational (as well as the linear systems of surfaces contained in it). From now on we will denote this variety with $\mu _3^{22}$, or simply $\mu$; it is the image of the system $\infty ^3$ of lines $\Sigma$ intersection of the 
$\infty ^4$ of chords of $R^4$ with a linear complex (which, for now, we will suppose to be of very general type, and in general position with respect to $R^4$)}.

\section{What is nowadays Fano's Last Fano?}

Where can we find Fano's last Fano in modern literature? In order to answer this question, we 
reformulate Fano's construction via a modern and rather subtle tool, namely the Hilbert scheme.

Let $S$ be a smooth projective surface and consider the Hilbert scheme which parametrizes its zero dimensional subschemes of length $2$. By Grothendieck's theory this is a projective scheme of dimension $4$ which  is usually denoted as $\Hilb^{2}(S)$, or simply $S^{[2]}$. One can also consider the Chow Scheme of the set of two points on $S$, namely $S^{(2)} : = (S \times S)/\sigma_2$, where $\sigma_2$ is the symmetric group of permutations of two elements. Consider then the natural morphism Hilb to Chow, $S^{[2]} \ra S^{(2)}$. 

A line bundle $L$ on $S$ induces the $\sigma_2$-equivariant line bundle $L^{\boxtimes 2}$ on $S\times S$, which descends to a line bundle $L^{(2)}$ on $S^{(2)}$, which in turn can be pulled back via the Hilb to Chow morphism to the line bundle $L^{[2]}$ on $S^{(2)}$.

\smallskip
For simplicity, we assume that $S$ has irregularity zero; we will use the following very well known theorems by Fogarty,  \cite{fogarty1} and  \cite{fogarty2}.

\begin{theorem} Let $S$ be a smooth projective surface with $q = h^1(\cO_S) = 0$. 

a) $\Hilb^{2}(S)= S^{[2]}$ is a smooth projective variety of dimension $4$ which resolves the singularities of the Chow Scheme via the Hilb to Chow morphism $S^{[2]} \ra S^{(2)}$.

b) $\Pic(S^{[2]}) = \Pic(S) \oplus \Z(B/2)$, where $Pic(S)$ is embedded in $\Pic(S^{[2]})$ via the above described map $L \ra L^{[2]}$ and $B$ is the locus of non reduced schemes, i.e. the exceptional divisor of the Hilb to Chow map (Corollary 6.3 in \cite{fogarty2}).

\end{theorem}

In the next Proposition, we specialise the results of Fogarty to the case $S = \pu \times \pu$ and we add, for this case, the description of the Nef (or Mori) Cone of $ (\pu \times \pu)^{[2]}$; for a proof we refer, for instance, to Theorem 2.4 in \cite{BC}.

\begin{proposition} Let $S = \pu \times \pu$. Denote by $\pi_i$ the two projections and $H_i := \pi_i^*(\cO_{\pu}(1))$ (the line bundles associated to the two fibers). 

For brevity, we denote by $\cH$ the Hilbert Scheme $ (\pu \times \pu)^{[2]}$.

a) $\cH$ is a smooth projective variety of dimension $4$ and $Pic(\mathcal H) = \Z (H_1^{[2]}) \oplus \ Z (H_2^{[2]})  \oplus \Z(B/2)$

b) The Nef Cone of $\cH$ is the simplicial cone spanned by $H_1^{[2]}, H_2^{[2]}$ and $H_1^{[2]} + H_2^{[2]} - (B/2)$.

\end{proposition}

Let us describe the maps associated to the nef bundles which span the Nef Cone. We will get six maps from $\cH$, one for each of the three extremal rays, and one for each of the three extremal faces.

The map associated with (a sufficiently high multiple) of $H_1$ (respectively $H_2$), call it $\psi_1 \colon \cH \ra M'$ (respectively $\psi_2 \colon \cH \ra M''$),  is a birational map which contracts the divisor $D_2 \subset \cH $ consisting of the zero cycles supported on the fibers $f_2$ of the second ruling (respectively on the fibres $f_1$ of the first ruling) of $ \pu \times \pu$. More precisely the map contracts all zero cycles on a fiber to a point. This divisor is clearly isomorphic to $\pu \times (\pu)^{[2]} = \pu \times \proj^2$ and it is contracted to a rational curve $\pu$ by contracting each $\proj^2$ to a point.  The two divisors are disjoint (a zero cycle of length two can be contained in at most one fiber) and they can be contracted simultaneously, this contraction correspond to the face of the Nef Cone joining the two rays, call it $\psi : \cH \ra Q$.

The map associated to $H_1^{[2]} + H_2^{[2]} - (B/2)$ is the Hilb to Chow map, $(\pu \times \pu)^{[2]} \ra (\pu \times \pu)^{(2)}$, which contracts the divisor $B$. We have moreover two natural map of fiber type, $(\pu \times \pu)^{(2)} \ra ( \pu)^{(2)} = \proj ^2$. and finally other two maps $\phi_1: M' \ra \proj^2$, $\phi_2: M'' \ra \proj^2$ making the following diagram commutative

\[
\xymatrix{
& \proj ^2&&M' \ar_{\phi_1}[ll] \ar^{}[dr]&\\
(\pu \times \pu)^{(2)}\ar[dr]\ar[ur]& &\cH \ar[ur]^{\psi_1} \ar[dr]^{\psi_2} \ar[ll] \ar[rr]^{\psi} \ar[lu] \ar[ld]& &Q\\
&\proj ^2&&M'' \ar_{\phi_2}[ll] \ar^{}[ur]&\\
}
\]

\medskip
Now we would like to give a concrete projective description of the above abstract varieties. In order to do this we take a smooth projective surface $S$ and we choose an embedding $S \hookrightarrow \proj ^N$.
Note that to a subscheme of length $2$ on $S$ we can now associate the unique line containing its image in $\proj ^N$.
In this way we can think at $\Hilb^2(S)$ as "the variety $\infty^4$ of its (i.e. of $S \hookrightarrow \proj ^N$) chords" used by Fano. 

Thus we have a natural map from $\Hilb^2(S)$ to the Grassmannian of lines in $\proj^N$, whose image is the variety of the lines that are secants or tangents to $S \subset \proj^N$; we further compose with the Pl\"ucker embedding of the Grassmannian, $Gr(1,N) \ra \proj^{{\frac{(N+1)N}2}  -1}$.

If $S \subset \proj ^N$ with $N\geq \dim S +2$ then the generic secant of $S$ is not $3$-secant. This follows easily, cutting with  general hyperplanes, by the classical so called trisecant lemma, which states that a  a nonsingular nondegenerate curve $C \subset \proj^r$, $r \geq 3$, admits only $\infty ^1$ trisecant lines (see for instance Cap. 7B in \cite{Mum}).

This implies that if $N \geq 4$ the total map $\Hilb^2(S) \ra \proj^{{\frac{(N+1)N}2}  -1}$ is a birational map (onto its image). 

If we embed $S=\pu \times \pu$ via the linear system $H_1 \otimes H_2^{\otimes 2}$ we get the normal rational scroll of degree $4$, $R^4 \subset \proj^5$, and {\bf the map $\Hilb^2(\pu \times \pu) \ra  M_4\subset \proj^{14}$  is exactly the one in Fano's paper}.

\smallskip
Let us consider first the special case $S=  \pu \times \pu$ and the embedding given by the complete linear system $H_1 \otimes H_2$, that is we embed $\pu \times \pu$ as a smooth quadric surface $Q_2\subset \proj ^3$. Note that that the secant lines fill up the whole Grassmannian $G(1,3)$, since every line in $\proj ^3$ is secant to any quadric surface. The Pl\"ucker embedding maps $G(1,3)$ into a (Klein) quadric $4-$fold $Q_4$ in $\proj^5$. 

Therefore we have a birational surjective map
$\cH \ra Q_4 \subset \proj^5$. We show that it is the map $\psi$ in the above diagram.

Indeed it contracts the two divisors $D_1$ and $D_2$ to two curves,  $C_1, C_2 \subset Q_4$, which describe in the Grassmannian the lines in the ruling. These curves are conics. In fact their degree is the number of lines in a ruling that meet a fixed general line: since a general line intersect a quadric in two points, it will meet exactly two fibers for each ruling. 

As said before, all non zero dimensional fibers of $\psi$ are isomorphic to $\proj^2$, in particular they all have the same dimension.
By a general result, see Corollary 4.11 in  \cite{AW}, $\psi : \Hilb^2(\pu \times \pu) \rightarrow Q \subset \proj^5$ is the blow up of the quadric $Q_4 \subset \proj ^5$ along two disjoint smooth conics, $C_1, C_2$.

\smallskip
Let us now go back to the Fano case, i.e. we embed $S=\pu \times \pu$ via the linear system $H_1 \otimes H_2^{\otimes 2}$ as the normal rational scroll of degree $4$, $R^4 \subset \proj^5$. 
The birational surjective map $\cH \ra M_4\subset \proj^{14}$ is the one studied by Fano. In this case the map contracts only one of the two above mentioned divisors, namely the one corresponding to the ruling in lines of $R^4$ which we denote as above with $D_1$. In other words the map is the map $\psi_1$ in the above diagram.

Thus $M_4$ is smooth and $\psi_1$ is the smooth blow-up along the transform  $\tilde C_2 \subset M_4$ of $C_2$. This gives a different, more geometric, proof of the smoothness of $M_4$.

The other divisor $E:= D_1$ remains isomorphically equal in $M_4$ and it can be contracted as a smooth blow-down to the curve $C_1 \subset Q_4\subset \proj^5$, $\nu \colon M_4 \ra Q_4$. We proved that $C_1$ is a smooth conic, here we can add the fact that it is not contained in any plane $\proj ^2 \subset Q_4\subset \proj^5$. In fact, if this were the case, all lines in the ruling of lines of $R^4$, which are parametrized by $C_1$, would be contained in a $\proj ^4 \subset \proj ^5$, and the same for $R^4$, a contradiction.

\smallskip
Let $H$ be the hyperplane bundle in $\proj ^5$; the formula for the canonical bundle of the blow up gives
$$ - K_ {M_4}= \nu^* (4H) - 2E = 2(\nu^* (2H) - E).$$

The line bundle $\cL := \nu^* (2H) - E$ is very ample; it embeds $M_4$ into $\proj ^{14}$ (the space of quadrics in $\proj ^5$ containing a conic has dimension $15$) as a Fano manifolds of index $2$ and genus $12$. 

Since $M_4$ is the blow-up of a quadric we have $\Pic(M_4) = \Z^2$, that is $M_4$ is not "prime".

On the other hand, the line bundle $\nu^* (H) - E$ is nef and it gives a map $\phi _1\colon M_4 \ra \proj ^2$ which is a quadric bundle fibration over $\proj ^2$. 

\medskip
Looking at the classification obtained by Mukai \cite{Mu1} of Fano $4$- folds of index $2$ (coindex $3$ in Mukai' notation) one can find $M_4$, given as the blow-up of a four dimensional quadric along a conic, as the only one of genus $12$ (Example 2); see also the paper \cite{Wi} with more detailed proofs. The classification was based on Conjecture (ES) which was later proved in \cite{Me}. 

\smallskip
Since $ - K_ {M_4}=  2 \cL$, a general hyperplane section in $\cL = \nu^* (2H) - E$ is a Fano $3-$fold, which we denote as Fano did with $M_3^{22}$.
$\cL$ embeds $M_4$ as the image of $Q_4$ by the rational map given by the quadric hypersurfaces through a general (=not contained in a plane) conic in $Q_4$, therefore the hyperplane section $M_3^{22}$ is obtained blowing up the conic in the intersection of $Q_4$ with another quadric containg the conic. 
This proofs that the $M_3^{22}$, Fano's last Fano, is the number $16$ in the Mori-Mukai list of Fano $3$-folds with Picard number $2$, see \cite[Table 2]{MM}. In fact they describe this case as the blow up of a prime Fano $3-$fold of degree $4$ in $\proj ^5$ along a conic; a prime Fano $3-$fold of degree $4$ in $\proj ^5$, according to Iskovskikh, is a complete intersection of two quadrics in $\proj ^5$.

We summarize the results of the section in the following.

\begin{proposition} The projective variety $M^{22}_4 \subset \proj ^{14}$ constructed by Fano is a smooth Fano $4$-fold of index $2$ which can be described also as the blow-up of a smooth hyperquadric $Q_4 \subset \proj ^5$ along a smooth conic not contained in any $\proj^2$ of the hyperquadric. This is Example 2 (5.(g=12)) in  Mukai's classification, \cite{Mu1}.

A general hyperplane section of $M^{22}_4$ is a smooth (non prime) Fano $3$-fold, denoted by Fano as $\mu:= M^{22}_3 \subset \proj ^{13}$, which can be constructed as the blow up of a complete intersection of two quadrics in $\proj ^5$ along a conic. This is number $16$ in Mori-Mukai classification, see \cite[Table 2]{MM}.
\end{proposition}

\begin{remark} One can embed $S=\pu \times \pu$ via the linear system $H_1 \otimes H_2^{\otimes m}$, with $m \geq 2$, to get the normal rational scroll $R^{2m}$ in $\proj^{2m+1}$. 
In this case, the birational map $\Hilb^2(\pu \times \pu) \ra  M_4\subset \proj^{(m+1)(2m+1)-1}$ is given by the contraction of one divisor,  corresponding to the ruling in lines of $R^{2m}$. The immersion $M_4\subset \proj^{(m+1)(2m+1)-1}$ is not given by the fundamental line bundle.

If we embed $S=\pu \times \pu$ via the linear system $H_1^{\otimes l} \otimes H_2^{\otimes m}$, with $l,m \geq 2$, the above construction gives simply different embeddings of $\Hilb^2(\pu \times \pu)$.

\end{remark}

\end{document}